\documentclass{elsarticle}
\usepackage{amscd,enumerate}
\usepackage{amssymb,amsmath,amsthm}
\usepackage{latexsym,amsfonts,amscd}
\usepackage{dsfont} 

\usepackage[utf8]{inputenc}
\usepackage[OT4]{fontenc}

\newcommand*{\R}{\ensuremath{\mathbb{R}}}

\newcommand*{\N}{\ensuremath{\mathbb{N}}}
\newcommand*{\Z}{\ensuremath{\mathbb{Z}}}
\newcommand*{\borel}{\ensuremath{\mathfrak{B}}}

\newtheorem{theorem}{Theorem}
\newtheorem{definition}[theorem]{Definition}
\newtheorem{corollary}[theorem]{Corollary}
\newtheorem{lemma}[theorem]{Lemma}

\newtheorem{proposition}[theorem]{Proposition}

\begin{document}
\title{How fast increasing powers of a continuous random variable
converge to Benford's law}
\author{Michał Ryszard Wójcik}
\address{Institute of Geography and Regional Development,
University of Wrocław\\
pl. Uniwersytecki 1, 50-137 Wrocław, Poland\\
telephone: +48 71 375 22 44,
email: michal.wojcik@uni.wroc.pl}

\begin{abstract}
It is known that
increasing powers of a continuous random variable
converge in distribution to Benford's law as the exponent approaches infinity.
The rate of convergence has been estimated using Fourier analysis,
but we present an elementary method, which is easier to apply
and provides a better estimation in the widely studied case
of a uniformly distributed random variable.
\end{abstract}


\maketitle

\section{Introduction}

For a fixed choice of base $b\in(1,\infty)$,
let the significand of a positive number $x>0$ be defined as
$S_b(x)=xb^{-\lfloor\log_b x\rfloor}$
in accordance with Definition 2.3 in \cite{basictheory}.
It is well-known that if $X$ is a positive continuous random variable
with any density then $S_b(X^n)$ converges in distribution
as $n\to\infty$ to Benford's law for base $b$,
which is equivalently stated as $n\log_bX\bmod{1}$ converges in distribution
to $U[0,1)$ as $n\to\infty$, e.g.~\cite[Th.~4.17]{basictheory},
\cite[Th.~3]{appliedfourier}, \cite{lolbert}.
This fact has been proposed as an explanation why population statistics
conform to Benford's Law, e.g.~\cite{ross}, \cite{population}.

For practical applications and reliable insights into natural phenomena
the concept of convergence in distribution is too abstract.
We need a precise estimation of the rate of converge and a tangible
measure of the distance between a given power of the initial random variable
and its limit. This is provided by the concept of the total variation distance
between two probability measures.

Let $X,Y\colon\Omega\to\R$ be two random variables.
Let $$\delta(X,Y)=\sup\big\{|P(X\in A)-P(Y\in A)|\colon A\in\borel(\R)\big\}.$$
If $X$ and $Y$ have densities $f$ and $g$ then $\delta(X,Y)=
\frac{1}{2}\int_{-\infty}^\infty|f(x)-g(x)|dx.$

Note that convergence according to this distance is stronger than
convergence in distribution, and if the rate of convergence in terms of this distance
is provided then we have the desired tangible estimation.

\section{Presentation of the method}

We shall estimate
\begin{equation}
\delta\big(X\bmod{1},U[0,1)\big)=
\sup_{A\in\borel([0,1))}\big|P(X\bmod{1}\in A)-\lambda(A)\big|
\end{equation}
in terms of the total variation of the density of $X$.

\begin{lemma}\label{keylemma}
Let $X$ be a real-valued random variable with density $f$.
Then $$\delta\big(X\bmod{1},U[0,1)\big)\le\frac{1}{2}
\sum_{k\in\Z}\int_{k}^{k+1}|f(x)-s_k|dx$$
where $s_k=\int_k^{k+1}f(t)dt.$
\end{lemma}
\begin{proof}
Let $s(x)=\sum_{k\in\Z}\mathds{1}_{[k,k+1)}(x)\int_k^{k+1}f(t)dt.$
Then $s$ is the density of some random variable $S$
with $S\bmod{1}\sim U[0,1)$.
Thus for any $A\in\borel([0,1))$,
$\big|P(X\bmod{1}\in A)-\lambda(A)\big|
=\big|\int_{E}f(x)dx-\int_{E}s(x)dx\big|,$
where $E=\bigcup_{k\in\Z}(A+k)$. Therefore,
$\delta\big(X\bmod{1},U[0,1)\big)$ does not exceed
$$\sup_{E\in\borel(\R)}\Big|\int_{E}f(x)dx-\int_{E}s(x)dx\Big|=
\frac{1}{2}\int_{-\infty}^\infty|f(x)-s(x)|dx.$$
\end{proof}

\begin{proposition}\label{keyinequalities}
Let $f\colon[a,b]\to[c,d]$ be integrable and let $y=\frac{1}{b-a}\int_a^bf(x)dx$.
Then\begin{equation}\label{banal1}
\int_a^b|f(x)-y|dx\le\frac{(b-a)(d-c)}{2}.\end{equation}
Moreover, if $f$ is monotonic and convex then\begin{equation}\label{banal2}
\int_a^b|f(x)-y|dx\le\frac{(b-a)(d-c)}{4}.\end{equation}
\end{proposition}
The estimation in (\ref{banal1}) cannot be improved because
in the worst possible case we get equality if $f$ assumes exactly two values
on sets of equal measure.
The estimation in (\ref{banal2}) cannot be improved because
in the worst possible case we get equality when $f$ is a straight line.

\begin{definition}\label{totalvariation}\textnormal{
For $f\colon\R\to\R$, let the total variation of $f$ restricted to
the minimal integer-delineated interval containing its support be defined as
$$TV(f)=\sup\Big\{\sum_{n\in\N}|f(x_{n+1})-f(x_n)|\colon
\{x_n\}_n\textnormal{ is an increasing sequence in }(n,m)\Big\},$$
where
$$n=\inf\big\{k\in\Z\colon f\big([k,k+1)\big)\not=\{0\}\big\},$$
$$m=\sup\big\{k\in\Z\colon f\big((k-1,k]\big)\not=\{0\}\big\}.$$
It may happen that $n=-\infty$ or $m=\infty$.}
\end{definition}

For example, if $f\colon\R\to\R$ is the density of $U[0,1)$, then $TV(f)=0$,
while its total variation on the interval $[0,2]$ equals 1
and its total variation on the interval $[-1,2]$ equals 2.

Note that if $f\colon\R\to\R$ is monotonic on $(n,m)$, where
$n\in\Z\cup\{-\infty\}$, $m\in\Z\cup\{\infty\}$,
then $TV(f)=\sup f\big((n,m)\big)-\inf f\big((n,m)\big)$.
\begin{theorem}\label{maintheorem}
Let $X$ be a real-valued random variable with density $f$.
Then\begin{equation}
\delta\big(X\bmod{1},U[0,1)\big)\le\frac{TV(f)}{4}.\end{equation}
Moreover, if $P(n<X<m)=1$ with $n\in\Z\cup\{-\infty\}$,
$m\in\Z\cup\{\infty\}$
and $f$ is monotonic and convex on the interval $(n,m)$, then
\begin{equation}\label{convexboost}
\delta\big(X\bmod{1},U[0,1)\big)\le
\frac{\sup f\big((n,m)\big)-\inf f\big((n,m)\big)}{8}.\end{equation}
\end{theorem}
\begin{proof}
Immediate from Lemma \ref{keylemma},
Proposition \ref{keyinequalities} and Definition \ref{totalvariation}.
\end{proof}

\begin{corollary}\label{metoda}
Let $X$ be a real-valued random variable with a density $f$ such that $TV(f)<\infty$.
Then $nX\bmod{1}\to U[0,1)$ in distribution as $n\to\infty$.
Moreover,
$$\delta\big(nX\bmod{1},U[0,1)\big)\le\frac{TV(f)}{4n}
\textnormal{ \ \ for all }n\in\N.$$
\end{corollary}
\begin{proof}
The density of $nX$ is $f_n(x)=f(x/n)/n$, so $TV(f_n)=TV(f)/n$.
\end{proof}

Note that we are not limited to integers.
In fact, Corollary \ref{metoda} could be stated so that
$aX\bmod{1}\to U[0,1)$ in distribution as $a\to\infty$
through all real values $a$, but the integer-delineation notion of $TV(f)$
would have to be abandoned and replaced with the more general notion
$$TV'(f)=\sup\Big\{\sum_{n\in\N}|f(x_{n+1})-f(x_n)|\colon
\{x_n\}_n\textnormal{ is an increasing sequence in }\R\Big\},$$
which gives double the value of the original $TV(f)$
in certain naturally occurring cases.

\section{Application of the method to a classical case}

We shall test the accuracy of our method by applying it to the classical
case of $X\sim U[1,10]$ and its integer powers $X^n$,
which has often been taken up in the literature on Benford's law,
e.g.~\cite{sarkar}, \cite{turner} and \cite{appliedfourier}.

\begin{theorem}\label{mojeszacowanie}
Let $b>1$. Let $Y\sim U[1,b]$ and $X=\log_bY$.
Then \begin{equation}\label{gdfdtftygvvd}
\delta\big(nX\bmod{1},U[0,1)\big)\le\frac{\ln b}{8n}
\textnormal{ \ \ for all }n\in\N.\end{equation}
\end{theorem}
\begin{proof}
If $n\in\N$, then 
$nX=n\log_bY$ has density $$f(x)=\frac{\mathds{1}_{[0,1]}(x/n)}{n}\frac{\ln b}{b-1}b^{x/n},$$
which is increasing and convex on $[0,n]$ with
$$\frac{\ln b}{n(b-1)}\le f(x)\le \frac{\ln b}{n(b-1)}b\text{ \ \ for all }x\in[0,n].$$
The proof is finished by invoking Theorem \ref{maintheorem}(\ref{convexboost}).
\end{proof}

Let us compare the estimation in Theorem \ref{mojeszacowanie}
with the Fourier series approach presented by Jeff Boyle, \cite{appliedfourier}.

\section{Comparison of the method with a Fourier analysis approach}

\begin{proposition}
Let $X\colon\Omega\to[0,1)$ be a random variable with density $f$.
Suppose that $f\in L^2([0,1])$.
Let $n\in\N$.
Then$$\delta\big(nX\bmod{1},U[0,1)\big)\le
\frac{1}{2}\sqrt{\sum_{k\in\Z\setminus\{0\}}|\hat{f}(nk)|^2}$$
where $\hat{f}(k)$ are the Fourier coefficients of $f$, that is
$\hat{f}(k)=\int_0^1f(x)e^{-2\pi ikx}dx$.
\end{proposition}
\begin{proof}
Let $f_n$ be the density of $nX\bmod{1}$.
By Hölder's inequality,
$$\delta\big(nX\bmod{1},U[0,1)\big)=
\frac{1}{2}\int_0^1|f_n(x)-1|dx
\le\frac{1}{2}\sqrt{\int_0^1|f_n(x)-1|^2dx}.$$
By Parseval's formula,
$$\int_0^1|f_n(x)-1|^2dx=\sum_{k\in\Z\setminus\{0\}}|\hat{f_n}(k)|^2,$$
where $\hat{f_n}(k)=Ee^{2\pi i k(nX\bmod{1})}=Ee^{2\pi i(nk)X}=\hat{f}(nk)$
for each $k\in\Z$.
\end{proof}

\begin{corollary}[Jeff Boyle, \cite{appliedfourier}]\label{fourierszacowanie}
Let $b>1$. Let $Y\sim U[1,b]$ and let $X=\log_bY$.
Then \begin{equation}\label{gdfdtygvncbcgtr}
\delta\big(nX\bmod{1},U[0,1)\big)<\frac{\ln b}{2\sqrt{12}n}
\textnormal{ \ \ for all }n\in\N.\end{equation}
\end{corollary}
\begin{proof}
$X$ has density $$f(x)={\mathds{1}_{[0,1]}(x)}\frac{\ln b}{b-1}b^{x}$$
with $$\hat{f}(k)=\frac{\ln b}{\ln b-2\pi in}\text{ \ \ for each } k\in\Z.$$
Then
$$\sum_{k\in\Z\setminus\{0\}}|\hat{f}(nk)|^2<
\sum_{k\in\Z\setminus\{0\}}\frac{(\ln b)^2}{|2\pi ikn|^2}
=\frac{(\ln b)^2}{4\pi^2n^2}\sum_{k\in\Z\setminus\{0\}}\frac{1}{k^2}
=\frac{(\ln b)^2}{4\pi^2n^2}\cdot 2\frac{\pi^2}{6}.$$
\end{proof}

Our Theorem \ref{mojeszacowanie} gives a better estimation
than Jeff Boyle's Fourier analysis approach of Corollary \ref{fourierszacowanie}
because $2\sqrt{12}<8$.
Let us calculate the exact value of $\delta\big(nX\bmod{1},U[0,1)\big)$
in (\ref{gdfdtftygvvd}) and (\ref{gdfdtygvncbcgtr})
in order to assess the accuracy of these estimations.

\section{Comparison of the method with exactly computed values}

\begin{theorem}\label{porachowano}
Let $a>0$ and $b>1$. Let $X\sim U[1,b)$.
Then
\begin{equation}\label{hgd6t43fd5}
\delta\big(\log_bX^a\bmod{1},U[0,1)\big)=
\frac{u\ln u-u+1}{b^{1/a}-1},\end{equation}
where $$u=\frac{b^{1/a}-1}{\ln b^{1/a}}.$$
\end{theorem}
\begin{proof}
Let $Y=\log_bX^a\bmod{1}$. Let $F(t)=P(Y\le t)$. Let $x=b^{1/a}$.
Elementary calculations involving the sum of a finite geometric sequence
yield $$F(t)=\frac{x^t-1}{x-1}\text{ \ \ \ for all }t\in[0,1].$$
$Y$ has density
$$f(t)=\mathds{1}_{[0,1]}(t)\Big(\frac{\ln x}{x-1}\Big)x^t
\text{ \ \ \ for all }t\in\R,$$
which is continuous and strictly increasing on $[0,1]$.
Therefore,
$$\delta\big(Y,U[0,1)\big)=\int_0^{t_0}(1-f(t))dt=t_0-F(t_0),$$
where $f(t_0)=1$.
Since $u=\frac{x-1}{\ln x}$, $f(\log_xu)=1$,
$$\log_xu-F(\log_xu)=\frac{\ln u}{\ln x}-\frac{u-1}{x-1}=
\frac{u\ln u}{x-1}-\frac{u-1}{x-1}.$$
\end{proof}

The following table shows that the estimation (\ref{gdfdtftygvvd})
in Theorem \ref{mojeszacowanie} is nearly as good as the exact value
(\ref{hgd6t43fd5}) calculated in Theorem \ref{porachowano}
and perceptibly better than the estimation (\ref{gdfdtygvncbcgtr})
provided by the Fourier analysis method.

\begin{center}\begin{tabular}{c|c| c| c}
n & exact (\ref{hgd6t43fd5}) & mine (\ref{gdfdtftygvvd})
& Fourier (\ref{gdfdtygvncbcgtr})\\
\hline
1&0.2688434&0.2878231&0.3323495\\
2&0.1413379&0.1439116&0.1661748\\
3&0.0951662&0.0959410&0.1107832\\
4&0.0716270&0.0719558&0.0830874\\
5&0.0573959&0.0575646&0.0664699\\
8&0.0359366&0.0359779&0.0415437\\
10&0.0287611&0.0287823&0.0332350\\
20&0.0143885&0.0143912&0.0166175\\
50&0.0057563&0.0057565&0.0066470\\
100&0.0028782&0.0028782&0.0033235\\
1000&0.0002878&0.0002878&0.0003323
\end{tabular}\end{center}

\section{Concluding remarks}

The presented method uses only the basic concepts of probability theory
and all its inequalities can be understood in terms of the areas under
graphs of functions of one variable with elementary geometrical estimations.
The accuracy is almost perfect in the tested case and perceptibly better
than the Fourier analysis approach.
There is no need to compute the Fourier coefficients of the density function.
The total variation of the density function is easily computed in the case
of monotonic or unimodal densities, and does not exceed twice
the maximum value of the density.

\end{document}